\documentclass[12pt]{article}

\usepackage[pdftex]{graphicx}
\usepackage{hyperref}
\usepackage{makeidx}
\usepackage[vcentermath]{youngtab}
\usepackage{young}
\usepackage{latexsym}
\usepackage{amsmath}
\usepackage{amssymb}
\usepackage{amsfonts}
\usepackage{eucal}
\usepackage{cite}
\usepackage{enumerate}
\usepackage{amsthm}
\usepackage[pdftex]{graphicx}
\usepackage{latexsym}
\usepackage{amsmath}
\usepackage{amssymb}
\usepackage{amsfonts}
\usepackage{eucal}
\usepackage{enumerate}
\usepackage{amsthm}
\usepackage{bbm}
\usepackage{pgf}
\usepackage{pgffor}
\usepackage{epstopdf}
\usepgfmodule{shapes}
\usepgfmodule{plot}
\usepackage{tikz}

\usepackage{fullpage}

\usetikzlibrary{arrows, petri, topaths}
\usepackage{tkz-berge}

\allowdisplaybreaks[1]
\theoremstyle{plain}
\newtheorem{coro}{Corollary}[section]
\newtheorem{lem}[coro]{Lemma}

\newtheorem{theo}[coro]{Theorem}

\theoremstyle{definition}
\newtheorem{defi}[coro]{Definition}
\newtheorem{exa}[coro]{Example}

\theoremstyle{remark}
\newtheorem{rem}[coro]{Remark}

\newcommand{\dtg}{\mathrm OTG}

\renewcommand{\Pi}{{\textbf v}}

\newcommand{\R}{\mathbb{R}}

\newcommand{\wout}{{\backslash}}
\renewcommand{\phi}{{\varphi}}
\newcommand{\lb}{\left\{}
\newcommand{\rb}{\right\}}

\newcommand{\mylabel}[1]{\label{#1}}

\begin{document}

\title{Oriented Threshold Graphs}
\author{Derek Boeckner \\ {\footnotesize\emph{University of Nebraska - Kearney}}\\{\footnotesize{\emph{ boecknerdc@unk.edu}}}}
\maketitle
\begin{abstract}

Threshold graphs are a prevalent and widely studied class of simple graphs. They have several equivalent definitions which makes them a go-to class for finding examples and counter examples when testing and learning.  This versatility has led to many results about threshold graphs and similar structures.  We look to generalize this class of graphs to oriented graphs (directed simple graphs.)  We give generalizations to four of the most versatile definitions and show their equivalence in the oriented case.  We finish with a proof enumerating the number of these oriented threshold graphs which relates to the Fibonacci numbers.  

\end{abstract}






\section{Introduction}
\setcounter{equation}{0}

\subsection{History}
Threshold graphs were first seen in several publications in the mid 1970s.  Papers in a variety of areas independently developed basic definitions for a class of graphs which gets it's name from a 1973 paper titled \emph{Set-packing Problems and Threshold Graphs} by Chv\'atal and Hammer \cite{CH73}.   These graphs have been found in numerous applications since their introduction.  These cover a wide range of subjects including applications in set-packing, parallel processing, resource allocation, scheduling, and psychology.  There is a great introduction to threshold graphs and their applications in the book \emph{Threshold Graphs and Related Topics} by Mahadev and Peled, \cite{MP95}.

In recent years, the limit points of threshold graphs (as graphons) have been studied in a paper by Diaconis, Holmes, and Jansen, \cite{DHJ08}.  This gives an interesting result that their limits can be realized as $\lb 0, 1\rb-$valued increasing functions on the unit square.   

Another recent result is by Cloteaux, LaMar, Moseman, and Shook, \cite{CLMS12}.  Their generalization to directed graphs is focused on degree sequences and unique realizations.   This work is extended by Reilly, Scheinerman, and Zhang, \cite{RSZ14}.  These extensions generalize definitions of simple threshold graphs into directed graphs and demonstrate their equivalence with the definitions of Cloteaux, et al.  These definitions deal predominately with directed graphs in which 2-cycles (multiedges in the underlying graph) are permitted in order to obtain unique realizability. 

In this paper we'll look at oriented simple graphs where we prohibit such 2-cycles and see some surprisingly lovely results. 

\subsection{Background}

Mahadev and Peled in \cite{MP95} give a thorough treatment of the class of threshold graphs.  Here we give the basic definition and some equivalences. 

\begin{defi}\mylabel{TG} Let $G$ be a graph.  We say that $G$ is a\index{graph, threshold} \emph{threshold graph} if there exists a threshold $t\in \R$ and a vertex weight function $w:V(G)\to \R$  such that $e=(x,y)\in E$ if and only if $w(x)+w(y) > t$.

\end{defi}

Though this is a fairly simple definition to work with, there are several equivalences that will be worth considering.  To understand them we need two definitions: 

\begin{defi} \mylabel{split} \index{graph, split}A graph, $G=(V,E)$, is said to be \emph{split} if the vertex set $V$ can be partitioned into two classes $K$ and $I$ such that $K$ induces a clique in $G$, and $I$ is an independent set in $G$. \end{defi}

\begin{defi} A graph, $G=(V,E)$,  on four vertices is a \emph{switch} \index{switch}if there is an ordering of the vertices, $a,b,c,d$ such that $ab,cd\in E$ and $ad,bc\notin E$. 

We say a graph is \emph{switch-free} if it contains no induced switches.  \end{defi}

The definition of a switch describes several different graphs. For example, $C_4$, 2 copies of $K_2$, and $P_3$ the path with three edges are all switches. 

We can now state 4 characterizations of threshold graphs.

\begin{theo}\mylabel{TGequiv} \cite{MP95} The following are equivalent:

\begin{enumerate}[(i)]

\item $G$ is a threshold graph.

\item $G$ is a split graph and the vertex neighborhoods are nested.

\item $G$ is switch-free.

\item The graph $G$ can be constructed by starting with a single vertex and sequentially adding either a dominating vertex or an isolated vertex at each step.

\end{enumerate}

\end{theo}

Equivalence (iv) of Theorem \ref{TGequiv} allows a very nice constructive bijection between binary sequences of length $n-1$ and threshold graphs on $n$ vertices.   We define a threshold graph by  creating a binary sequence $\bar s =(s_i)_{i=1}^n$ where $s_i=1$ if the vertex added is dominating, or $s_i=0$ if it is isolated.  Given such a sequence we define $T(\bar s)$ to be the threshold graph associated with it.  In this construction the very first vertex is both isolated and dominating; we therefore classifying it as a 0 or 1 is somewhat misleading.  We will always classify the first vertex as $\star$ when giving a threshold graph in its sequential form.  We use the convention that the sequence is constructed right to left, thinking of the first vertices added as least significant, as in least significant digits in a number.  

\section{Oriented Threshold Graphs}\mylabel{dtgsec}

There are several definitions for a threshold graph in the undirected case, Theorem \ref{TGequiv}.  We  begin by developing an analogous vocabulary for oriented graphs and then state a theorem presenting several equivalent definitions of an oriented threshold graph.   

\begin{defi}\mylabel{dtgdef} \index{digraph, threshold}An oriented graph, $G=(V,E)$, is said to be \emph{threshold} if there exists a weight function on the vertices $w:V\to \R$ and a threshold value $t\in\R$ such that $\overrightarrow{xy}\in E$ if and only if $|w(x)|+|w(y)| \ge t$ and $w(x)>w(y)$.\end{defi}

Since we insist on strict inequality, $w(x)>w(y),$ so that our oriented threshold graphs contain no loops.  This also lets us think of edges running `downhill.'

\subsection{Background and Oriented Threshold Equivalence}

Before we state our main theorem which is directly analogous to Theorem \ref{TGequiv}, we need to develop vocabulary to state corresponding statements in the oriented case. 

The first generalization we'll explore is \emph{nested neighborhoods}.  These next definitions will help us generalize the concept of nested neighborhoods to oriented graphs where we have not just a total neighborhood, but have in and out neighborhoods as well. 

\begin{defi}\mylabel{displit} \index{digraph, split}\index{displit}A directed graph is said to be \emph{displit} if the vertex set can be partitioned into three classes, $V=B\cup I\cup T$ (Bottom, Independent, and Top), with the properties:

\begin{enumerate}[i)] 

\item $I$ is an independent set,

\item the graph induced by $B\cup T$ is a tournament,

\item all edges between $T$ and $B\cup I$ are directed from $T$, 

\item all edges between $B$ and $T\cup I$ are directed into $B$. 

\end{enumerate}

\end{defi}

A small example of a displit graph may be helpful in understanding this definition. 

\begin{figure}[htb]\caption{A small displit graph}\mylabel{displitexa}

\begin{center}

\begin{tikzpicture}[scale=.75, transform shape, >=triangle 45]

\node[place] (2) at (-2,2){T};
\node[place] (1) at (-1,1){T};
\node[place] (0) at (0,0){I};
\node[place] (-1) at (-1,-1){B};
\node[place] (-2) at (-2,-2){B};
\node[place] (i) at (1.4,0){I};
\node[place] (ii) at (2.8,0){I};

\tikzstyle{EdgeStyle}=[->,bend right]

\Edge({2})({-1})
\Edge({2})({-2})
\Edge({1})({-2})

\tikzstyle{EdgeStyle}=[->]
\Edge({2})({1})

\Edge({1})({0})
\Edge({1})({-1})

\Edge({0})({-1})

\Edge({-1})({-2})

\tikzstyle{EdgeStyle}=[->, bend left]
\Edge({2})({ii})
\Edge({2})({i})
\Edge({1})({i})
\Edge({i})({-2})
\Edge({2})({0})
\Edge({0})({-2})

\end{tikzpicture} 

\end{center}

\end{figure}

LaMar prior to working on the unigraphic sequence problem for digraphs gave a definition for split digraphs in \cite{L12}, the definition we just gave fits within his, but is slightly more restrictive.  With this stronger definition, we end up with a smaller class of graphs, but we are able to say much more about the structure of our class, both by giving an inductive construction from a ternary sequence, and by defining what it means to have nested neighborhoods in the sense of directed threshold graphs.

\begin{defi}\mylabel{doms} \index{in-dominated}\index{out-dominating}A vertex, $v$, in a digraph $D$ is called an \emph{out-dominating} (\emph{in-dominated}) if it is adjacent to every other vertex in $D$ and is a source (resp. sink). \end{defi}

\begin{defi} Let $\sigma:V\to 2^V$ be a function from a set to its power set.  We say the function $\sigma$ is \index{nested}\index{function, nested}\emph{nested}, on $S\subset V$ if for every $x,y\in S$ we have $\sigma(x)\subseteq \sigma(y)\cup \lb y\rb$ or $\sigma(y)\subseteq \sigma(x)\cup\lb x\rb$.  We denote $\sigma(x)\subseteq \sigma(y)\cup\lb y\rb$ by $x\unlhd^\sigma y$.   

We say the function $\sigma$ is {\emph strictly nested} on $S\subset V$, if for every $x,y\in S$ we have either $\sigma(x)\subseteq \sigma(y)$ or $\sigma(y)\subseteq \sigma(x)$.  For $x,y\in V$ we denote $\sigma(x)\subsetneq \sigma(y)$ by $x\lhd^\sigma y$.  

If $\sigma$ is nested (resp. strictly nested) on all of $V$, we say that $\sigma$ is nested (resp. strictly nested.) \end{defi}

Recall that in Theorem \ref{TGequiv} one of the conditions that makes a graph threshold is that it is split and has nested neighborhoods.  That is, the neighborhood function of a threshold graph is a nested function.  With this in mind, the proper way to view nested and strictly nested functions in terms of directed threshold graphs is via the following definition.

\begin{defi} \mylabel{pnn} Let $D$ be a displit graph with clique $K=T\cup B$ and independent set $I$.  We say $D$ has \index{ properly nested neighborhoods} \emph{properly nested neighborhoods} if the following hold for $N,N^+, N^-: V\to P(V)$ the neighborhood, out-neighborhood and in-neighborhood functions: 

\begin{enumerate}[i)]

\item $N$ is nested,

\item $N^+$ and $N^-$ are nested on $I$ and for $x,y\in I$ we have if $x\unlhd^N y$ then $x\unlhd^- y$ and  $x\unlhd^+ y$ (I've suppressed the Ns in the inequalities to make the notation less cumbersome.)

\item $N^+$ and $N^-$ are strictly nested on $K$ and for $x,y\in K$ we have $x\lhd^+ y$ if and only if $y\lhd^- x$.  

\end{enumerate}

\end{defi}

The properly nested neighborhoods condition states that the total neighborhoods are nested.  For vertices in $I$, the size of in and out neighborhoods are directly correlated, whereas in $B\cup T$ the in and out neighborhoods are inversely correlated.  Figure \ref{displitexa} gives an example of a graph which has properly nested neighborhoods. 

Now, we can generalize Theorem \ref{TGequiv} to directed threshold graphs.



\begin{theo}\mylabel{equivdef} The following are equivalent for a graph $G=(V,E)$: 

\begin{enumerate}[(a)]

\item $G$ is an oriented threshold graph. 

\item $G$ is a transitive orientation of a threshold graph. 

\item $G$ is a displit graph and has properly nested neighborhoods.

\item $G$ can be constructed from the one vertex empty graph by successively adding a independent vertex, an out-dominating vertex or an in-dominated vertex. 

\end{enumerate}
\end{theo}

\begin{proof} \emph{($a\implies b$):} Let $w:V(G)\to\R$ be the weight function and $t\in \R$ the threshold value for $G$.  To show the underlying graph is threshold, use the weight function $|w(v)|$ and the same threshold $t$.  For transitivity, suppose $x\rightarrow y$ and $y\rightarrow z$.  By definition then, we know that $|w(x)|+|w(y)|\ge t$ and $|w(y)|+|w(z)|\ge t$ and  that $w(x)>w(y)>w(z)$.  We need to consider 2 cases to show that $x\to z$.

Case 1) $w(y)\ge 0$:  Then $w(x)>0$ so $|w(x)|\ge |w(y)|$ therefore $|w(x)|+|w(z)|\ge |w(y)|+|w(z)|\ge t.$  So $x\rightarrow z$. 

Case 2) $w(y)\le 0$:  Then $w(z)<0$ so $|w(z)|\ge |w(y)|$ therefore $|w(x)|+|w(z)|\ge |w(x)|+|w(y)|\ge t.$ So again $x\rightarrow z$.

This shows that $G$ is transitive, and thus a transitive orientation of a threshold graph.

\emph{($b\implies c$):} Let $G$ be the underlying threshold graph associated with the sequence $\bar s$ (Theorem \ref{TGequiv}.)  The initial vertex in the sequential construction is given the label $\star$.  Now, the collection of $\star$ and the $0$'s form an independent set; call this set of vertices $I$. The collection of $1$s form a clique, call this set $K$.  Set $T = N^-(\star)$ and $B= N^+(\star)$.  Since every $1$ was adjacent to $\star$, this partitions $K$.  By the transitivity of the ordering, if $t\in T$ and $b\in B$, then $t\to \star$ and $\star\to b$ so $t\to b$.   This gives us the partition $T\cup I\cup B.$  We have that all edges between $T$ and $B$ are oriented correctly.  

We still need to show that edges between $K$ and $I$ are oriented correctly. That is, we need to show the edges are directed from $T$ to $I$, and the edges are directed to $B$ from $I.$   To do this we first show that $N^+$ and $N^-$ are nested on $I$.  Since the underlying graph is threshold, the neighborhood function is nested (this is the nested neighborhoods condition of $(ii)$ in Theorem \ref{TGequiv}.)  So, let $i,j\in I$ with $i\unlhd^N j$.  Suppose there is $x\in N^+(i)\wout N^+(j)$.  Then $x\in N^-(j)$ since $i\unlhd^N j$.  But this means $i\to x \to j$ and transitivity gives $i\to j$, however, that is impossible as $i,j\in I$.  So we must have $N^+(i)\subseteq N^+(j)$.   A similar argument gives $N^-(i)\subseteq N^-(j)$.  This shows property $ii$) of the properly nested condition, Definition \ref{pnn}. 

Using this, and noting that  for all $i\in I$ we have $i\unlhd^N \star$, we get that $N^+(i)\subseteq N^+(\star)=B$ and $N^-(i)\subseteq N^-(\star)=T$, showing that the graph is displit. 

To show the third condition of properly nested neighborhoods, let $x,y\in K$ with $x\to y$.  Then by transitivity $N^+(y)\subsetneq N^+(x)$ (the inclusion is strict because $y\notin N^+(y)$) and $N^-(x)\subsetneq N^-(y)$ (again because $x\notin N^-(x)$) which completes all conditions.   

\emph{($c\implies d$):} Let $i\in I$ be minimal in $I$ with respect to total neighborhoods.   If $N(i)=\emptyset$ then it is an isolate.   If not, we have that either its in-neighborhood is non-empty or its out-neighborhood is non-empty.  Say $x\in N^+(i)$.  Then, $x\in N^+(j)$ for all $j\in I$ since the neighborhood function is nested on $I$.  Let $y$ be the maximum element in the order given by $N^-$ being strictly nested on $K$.   Then since $x\in N^+(j)$ for all $j\in I$, we get that $j\in N^-(x)\subsetneq N^-(y)$ for all $j\in I$.  This show $y$ is dominated by $I$.  To show $y$ is dominated by $K$, suppose to get a contradiction, $y\to z$ for some $z\in K$.  Then $y\in N^-(z)\subsetneq N^-(y)$ which is impossible.  This means $y$ is an in-dominated vertex.   A similar argument gives that if $x\in N^-(i)$ then the maximal vertex with respect to the strictly nested order on $N^+$ is an out-dominating vertex. 

In order to inductively choose an independent, out-dominating, or in-dominated vertex we must show that the removal of such a vertex leaves us with a transitive digraph with properly nested neighborhoods.  If the vertex is isolated, its removal has no effect on the neighborhoods or the transitivity of the graph.  If the vertex is dominating or dominated, then its removal decreases every neighborhood in exactly the same way, leaving comparability conditions intact.  The transitivity also remains as removal of a vertex in any transitive graph leaves the graph transitive. 

This gives us an inductive construction of the directed threshold graph as a sequence of independent, in-dominated, and out dominating vertices as required.

\emph{($d\implies a$)}:  The assumption gives a sequence of zeros, ones, and negative ones, say $(s_i)_{i=1}^n$.  If we forget (temporarily) about the sign on the ones, we have a sequence of zeros and ones corresponding to removing the direction on the edges.  This underlying graph is constructed by adding isolates or dominating vertices.  This means it is an unoriented threshold graph.  There is an injective weight function and threshold for this underlying graph \cite{MP95}, say $(w_i)_{i=1}^n$ and $t$.  It is enough to show then, that the weight function $$\vec w_i=\begin{cases} s_iw_i, & \mathrm{if} s_i\ne 0 \cr w_i, & \mathrm{if} s_i=0.\end{cases}$$ gives the correct orientation of the edges.  

Let's make an observation about the weights of the vertices.  Notice that the weight of a vertex is directly correlated to the size of its neighborhood.  This means the later in the sequence a $1$ happens, the higher the weight of the vertex associated to it.  Conversely, the later in the sequence a $0$ happens, the lower the weight of the associated vertex.   

With this, we see that the above weight function satisfies $|\vec w_i|>|\vec w_j|$ whenever $i>j$ and $|s_i|=1$.  This means that the orientation of the graph given by the above weight function is the same as the orientation given by the sequential construction. 

\end{proof}

\begin{rem}\mylabel{dtgseq} This last equivalence gives a ternary sequence which can be translated into an oriented threshold graph.  We call the graph a sequence $s$ produces, $\dtg(s)$

\end{rem}

\begin{exa} Let $\bar s = (1,-1,0,-1,\star)$.   The sequential construction yields the following graph.

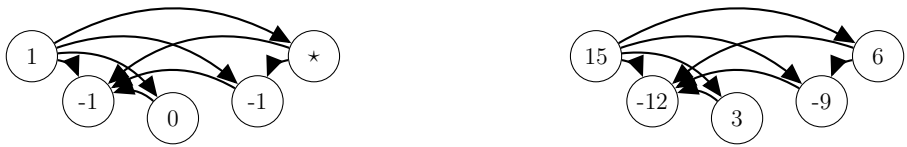
\begin{figure}[htb]\caption{(Left) $\dtg(1,-1,0,-1,\star)$ and (right) the directed threshold graph corresponding to the weight function $(15,-12, 3, -9, 6)$ and threshold $15$.}\mylabel{dtg1}

\begin{center}
\begin{tikzpicture}[scale=.75, transform shape,>=triangle 45]

\node[place] (5) at (-12.5,0){1};
\node[place] (-4) at (-11.5,-.8){-1};
\node[place] (1) at (-10,-1.1){0};
\node[place] (-3) at (-8.5,-.8){-1};
\node[place] (2) at (-7.5,0){$\star$};

\node[place] (50) at (-2.5,0){15};
\node[place] (-40) at (-1.5,-.8){-12};
\node[place] (10) at (0,-1.1){3};
\node[place] (-30) at (1.5,-.8){-9};
\node[place] (20) at (2.5,0){6};

\tikzstyle{EdgeStyle}=[->,bend left]
\Edge({5})({-4});
\Edge({5})({1});
\Edge({5})({-3});
\Edge({5})({2});
\tikzstyle{EdgeStyle}=[<-, bend left]
\Edge({-4})({1});
\Edge({-4})({-3});
\Edge({-4})({2});
\Edge({-3})({2});

\tikzstyle{EdgeStyle}=[->,bend left]
\Edge({50})({-40});
\Edge({50})({10});
\Edge({50})({-30});
\Edge({50})({20});
\tikzstyle{EdgeStyle}=[<-, bend left]
\Edge({-40})({10});
\Edge({-40})({-30});
\Edge({-40})({20});
\Edge({-30})({20});

\end{tikzpicture} 

\end{center}

\end{figure}

The vertex weights shown in the figure, $(15, -12, 3, -9, 6)$  with threshold $t=15$ also give the same graph. 
\end{exa}

\section{Sequential Form and Enumeration}

A few things before we go further: to draw and think about these oriented threshold graphs, the sequential definition is quite a bit more malleable; we work with it.  Recall that in the undirected case, the first vertex drawn is always an independent vertex and we denoted it by $\star$.  We'll use this convention with oriented threshold graph sequences as well.  Another simplification, instead of $+1$ and $-1$ we simply write $+$ and $-$ (resp.)  

Looking more closely at these sequences, things get a little messy.  In the undirected case, it is easy to just count the sequences, $\lb 0,1\rb^{n-1}$ ($n-1$ as the first vertex drawn does not matter.)  Things are a little more subtle in the case of directed threshold graphs.  Notice (in figure \ref{mppm}) that the sequence $(+-0\,\star)$ gives the same graph as the sequence $(-+0\,\star)$. The isomorphism switches the last two vertices, as shown in the following figure.

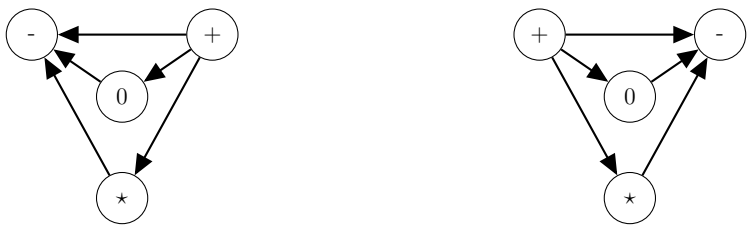
\begin{figure}[htb]\caption{(Left) $\dtg(-1,1,0,\star)$, and (right) $\dtg(1,-1,0,\star)$}\mylabel{mppm}

\begin{center}

\begin{tikzpicture}[scale=.75, transform shape,>=triangle 45]

\node[place] (+) at (-1.6,0){+};
\node[place] (-) at (1.6,0){-};
\node[place] (0) at (0,-1.1){0};
\node[place] (s) at (0,-2.9){$\star$};

\tikzstyle{EdgeStyle}=[->]

\Edge({+})({-})
\Edge({+})({0})
\Edge({+})({s})

\tikzstyle{EdgeStyle}=[<-]
\Edge({-})({0})

\Edge({-})({s})

\node[place] (-) at (-10.6,0){-};
\node[place] (+) at (-7.4,0){+};
\node[place] (0) at (-9,-1.1){0};
\node[place] (s) at (-9,-2.9){$\star$};

\tikzstyle{EdgeStyle}=[<-]

\Edge({-})({+})
\Edge({-})({0})
\Edge({-})({s})

\tikzstyle{EdgeStyle}=[->]
\Edge({+})({0})

\Edge({+})({s})

\end{tikzpicture} 

\end{center}

\end{figure}

\begin{lem}\mylabel{pmcomm} Given a sequence $\bar s:=(s_i)_{i=1}^n$, if there is a $k\in [n]$ such that $|s(k)|=|s(k-1)|$ then the sequence $s' = (s_1,s_2, \cdots, s_{k-2}, s_k, s_{k-1},s_{k+1},\cdots, s_n)$ produces a digraph isomorphic to the one produced by $s$.  \end{lem}

\begin{proof} Clearly if $s_k=s_{k-1}$ we're fine.  So without loss of generality assume $s_k=+$ and $s_{k-1}=-$.  So there is an edge $\overrightarrow{k(k-1)}$.  Now, just note the neighborhoods $N^+(k), N^-(k), N^+(k-1), N^-(k-1)$ do not change when we swap the order of $k$ and $k-1$, as the only edge affected is the one between them, and its order is switched as was needed. \end{proof}

\begin{rem}\mylabel{endplus} As the $\star$ at the beginning of any sequence can be thought of as a $+, -,$ or $0$, we can always think of $-$s adjacent to $\star$ as $+$'s.  \end{rem}

Using the previous lemma and remark we obtain a `canonical' representation for any isomorphism class, namely, $$(+^{p_l},-^{m_l}, 0^{z_l}, +^{p_{l-1}}, -^{m_{l-1}}, 0^{z_{l-1}}, \cdots, +^{p_1}, -^{m_1}, 0^{z_1}, +^{p_0} \star)$$ where $z_i\ne0$ for all $i$.   The notation $+^{p_i}$, $-^{m_i}$, and $0^{z_i}$ simply mean $p_i$ $+$s, $m_i$ $-$s, and $z_i$ $0$s (resp.)   

\begin{theo}\mylabel{canon} \index{digraph, threshold canonical form}There is a bijection between isomorphism classes of directed threshold graphs and sequences of the form $$(+^{p_l},-^{m_l}, 0^{z_l}, +^{p_{l-1}}, -^{m_{l-1}}, 0^{z_{l-1}}, \cdots, +^{p_1}, -^{m_1}, 0^{z_1}, +^{p_0}\star)$$ where the $z_i$ are positive integers and $p_i$ and $m_i$ are non-negative integers.  We call a sequence of this form canonical. 
\end{theo}

\begin{proof}  Let $G$ be a directed threshold graph. Then by Theorem \ref{equivdef} (d) there is a sequence of the characters $+,-$ and $0$ corresponding to $G$.  By Lemma \ref{pmcomm} any grouping of $+$s and $-$s can be rearranged so that all $+$s are to the left of all $-$s in the grouping without changing the isomorphism class.  Also, by Remark \ref{endplus}, if there is a grouping of $+$s and $-$s by $\star$ we can consider them as all $+$s.   Each of these groupings is separated by a grouping of $0$s.   This gives us a canonical sequence for the graph $G$.  

Now, every ternary sequence gives a unique representation of the form $$(+^{p_l},-^{m_l}, 0^{z_l}, +^{p_{l-1}}, -^{m_{l-1}}, 0^{z_{l-1}}, \cdots, +^{p_1}, -^{m_1}, 0^{z_1}, \star). $$   Since each graph gives a ternary sequence, if we can show  that two sequences that have different canonical forms give non-isomorphic graphs, we're done. 

So let $s$ and $t$ be two different ternary sequences in canonical form.  Let's get through a few trivial cases first.   If we take the underlying undirected graphs of $s$ and $t$, and they are non-isomorphic, then $s$ and $t$ themselves cannot be isomorphic.  To get to the undirected underlying graphs, we just look at the binary sequences where $+$s and $-$s are mapped to $1$, and $0$s are mapped to $0$.   If these aren't the same, we're done.  So assume, $s$ and $t$ have the same length and the same number of $0$s moreover the indices of the $0$s are the same.    Let $i$ be the leftmost index in which $s$ and $t$ differ, without loss of generality, say $s_i=+$ and $t_i=-$.  Let $k$ be the next index ($k<i$) where $s_k=t_k=0$ (this exists because $s$ and $t$ are in canonical form, and if it didn't, then $s$ and $t$ would be a sequence of $+$s.)  Then the grouping of $0$s that include $s_k$ and $t_k$ form an independent set, say $I$.  Now, all vertices in $I$ have the same degree, $n-k$.  

From here there are two cases: there is is another $+$ or $-$ after the grouping of $0$s, or $\star$ is the next vertex after the grouping of zeros in which index $k$ lies.  

In the first case, there are no other vertices of degree $n-k$ besides those in $I$.  In the sequence $s$, the number of vertices in the in-neighborhood of the vertices in $I$ is greater than they are in $t$.  This makes the two graphs non-isomorphic. 

In the other case, the number of $+$ and $-$ are different between the two sequences, meaning the displit partition of the vertices is different showing that the sequences represent non-isomorphic graphs. 

\end{proof}

Having a ternary canonical representation for oriented threshold graphs gives us an easy way to count the number of isomorphism classes of directed threshold graphs on $n$ vertices. 

\begin{theo}\mylabel{number} The number of isomorphism classes of directed threshold graphs on $n$ vertices is $F_{2n}$ the $2n$ Fibonacci number (where $F_0=0, F_1=1$.)

\end{theo}

\begin{proof} We find a recursion relation on the classes by looking at the sequences in canonical form.  We can always create a new sequence in canonical from one in canonical form by augmenting it with a $0$ or $+$, but only sequences that have a $0$ or $-$ can be augmented with a $-$ to form a new sequence in canonical form.   Let $T(n)$ be the number of sequences in canonical form, and $P(n)$ be the number of sequences in canonical form starting with a $+$.  Because a sequence in canonical form cannot have $-$ before a $+$, we get that $T(n) = T(n-1)+T(n-1)+(T(n-1)-P(n-1))$ where the first two terms come from augmenting a $0$ or $+$ to an old sequence, and the last term from augmenting a $-$.  Since we could always have augmented a sequence with a $+$, we get that $P(n)=T(n-1)$.   This gives us the recurrence $T(n)=3T(n-1)-T(n-2)$.  The initial conditions are that $T(1)=1$ (being the sequence $+$) and $T(2)=2$ (from the sequences $+\star$ and $0\star$.)  

Let's look at the Fibonacci sequence for a second.  Specifically, let's look at $F_{2n}$.  

\begin{eqnarray}  F_{2n} & =  F_{2n-1}+F_{2n-2}\cr & = 2F_{2n-2}+F_{2n-3} \cr & = 3F_{2n-2}-F_{2n-4} \cr & = 3F_{2(n-1)} - F_{2(n-2)} \end{eqnarray}

Also, notice that $F_0=1$ and $F_2=2$. Therefore we have the same recursion and starting values as the even Fibonacci numbers, and we're done.  \end{proof}

Putting this together with the characterization of directed threshold graphs, \ref{equivdef}, we see that the number of isomorphism classes of orientation of threshold graphs which are transitive is $F_{2n}$.  This follows directly from statement $b)$ from \ref{equivdef} states that every transitive ordering of a threshold graph is directed threshold.  This means that the number of transitive orientations come from tertiary sequences and the number of isomorphism classes we've just shown is $F_{2n}$. 

Canonical representation also allows us to count the number of non-isomorphic transitive orientations of a specific threshold graph.  

\begin{theo}\mylabel{ntransorient} Let $G$ be a threshold graph given by the sequence $(+^{p_l}, 0^{z_l},  \dots, +^{p_1}, 0^{p_1}, +^{p_0}, \star).$  The number of non-isomorphic transitive orientations of $G$ is $$\prod_{i=1}^l (p_l+1).$$\end{theo}

\begin{proof} An orientation of $G$ is given by turning some of the $+$s into $-$s.  Canonical form states that we get the same graph if we put $-$s at the end of the string of $+$/$-$s.  So, for each block of $+$s we simply have a choice of where to start putting $-$s.  There are $p_i+1$ choices for each block.  The last block doesn't get any $-$s.  The product of these choices is the total number of orientations which are transitive. \end{proof}


\bibliography{OTGbib}{}
\bibliographystyle{plain}

\end{document}